\documentclass[12pt, reqno]{amsart}
\usepackage{amsmath, amsthm, amscd, amsfonts, amssymb, graphicx, color}
\usepackage[bookmarksnumbered, colorlinks, plainpages]{hyperref}
\hypersetup{colorlinks=true,linkcolor=red, anchorcolor=green, citecolor=cyan, urlcolor=red, filecolor=magenta, pdftoolbar=true}
\usepackage{mathrsfs}
\newtheorem{theorem}{Theorem}[section]
\newtheorem{lemma}[theorem]{Lemma}

\theoremstyle{definition}

\theoremstyle{remark}

\numberwithin{equation}{section}

\begin{document}
\setcounter{page}{1}

\title[Jones Index Theorem revisited]{Jones Index Theorem revisited}

\author[Glubokov]{Andrey ~Yu. ~Glubokov$^1$}

\author[Nikolaev]{Igor ~V. ~Nikolaev $^2$}

\date{\today}

\address{$^{1}$ Department of Mathematics,  Purdue University, 
150 N. University Street, West Lafayette, IN 47907-2067,  United States.}
\email{\textcolor[rgb]{0.00,0.00,0.84}{agluboko@purdue.edu}}

\address{$^{2}$ Department of Mathematics and Computer Science, St.~John's University, 8000 Utopia Parkway,  
New York,  NY 11439, United States.}
\email{\textcolor[rgb]{0.00,0.00,0.84}{igor.v.nikolaev@gmail.com}}

%\dedicatory{All data is  available as part of this manuscript }



%\date{Received:  August 14, 2015; Revised: yyyyyy; Accepted: zzzzzz.}

\begin{abstract}
We prove the Jones Index Theorem using
the K-theory of a cluster $C^*$-algebra 
of the Riemann sphere with two boundary components.
\end{abstract}

\maketitle

%**************************************************************************
\section{Introduction}
%***************************************************************************
The Jones Index Theorem is an analog of the Galois theory for the von Neumann algebras
 [Jones 1991] \cite{J1}. 
Recall that the factor is a von Neumann algebra $\mathscr{M}$ with the trivial center. 
A subfactor $\mathscr{N}$ of the factor $\mathscr{M}$ is a subalgebra,  such  that  $\mathscr{N}$
is a factor. 
The index $[\mathscr{M}:\mathscr{N}]$ of a subfactor $\mathscr{N}$ of a type II factor $\mathscr{M}$ 
is  a positive real number $\dim_{\mathscr{N}}(L^2(\mathscr{M}))$,  where $L^2(\mathscr{M})$
is a representation of $\mathscr{N}$ obtained
from  the canonical trace on $\mathscr{M}$  using the Gelfand-Naimark-Segal (GNS) construction.
We refer the reader to  [Jones 1991] \cite[Section 2.5]{J1} for the missing definitions and details.  
The  Jones Index Theorem  says that such subfactors exist only if:
%*************************************************************************
\begin{equation}\label{eq1.1}
[\mathscr{M}:\mathscr{N}]  ~\in  ~[4, \infty) ~\bigcup  ~\{ 4\cos^2\left({\pi\over n}\right) ~|~n\ge 3\}.   
\end{equation}
%*************************************************************************    

\bigskip
The cluster algebra  $\mathscr{A}(\mathbf{x}, B)$ of rank $n$ 
is a subring of the field  of  rational functions in $n$ variables
depending  on a  cluster  of variables  $\mathbf{x}=(x_1,\dots, x_n)$
and a skew-symmetric matrix  $B=(b_{ij})\in M_n(\mathbf{Z})$ [Fomin \& Zelevinsky 2002]  \cite{FoZe1}.
The pair  $(\mathbf{x}, B)$ is called a  seed.
A new cluster $\mathbf{x}'=(x_1,\dots,x_k',\dots,  x_n)$ and a new
skew-symmetric matrix $B'=(b_{ij}')$ is obtained from 
$(\mathbf{x}, B)$ by the   exchange relations:
%*********************************************************************************************
\begin{eqnarray}\label{eq1.2}
x_kx_k'  &=& \prod_{i=1}^n  x_i^{\max(b_{ik}, 0)} + \prod_{i=1}^n  x_i^{\max(-b_{ik}, 0)},\cr 
b_{ij}' &=& 
\begin{cases}
-b_{ij}  & \mbox{if}   ~i=k  ~\mbox{or}  ~j=k\cr
b_{ij}+{|b_{ik}|b_{kj}+b_{ik}|b_{kj}|\over 2}  & \mbox{otherwise.}
\end{cases}
\end{eqnarray}
%******************************************************************************************* 

The seed $(\mathbf{x}', B')$ is said to be a  mutation of $(\mathbf{x}, B)$ in direction $k$,
where $1\le k\le n$;   the algebra  $\mathscr{A}(\mathbf{x}, B)$ is  generated by cluster  variables $\{x_i\}_{i=1}^{\infty}$
obtained from the initial seed $(\mathbf{x}, B)$ by the iteration of mutations  in all possible
directions $k$.   The Laurent phenomenon  says  that  $\mathscr{A}(\mathbf{x}, B)\subset \mathbf{Z}[\mathbf{x}^{\pm 1}]$,
where  $\mathbf{Z}[\mathbf{x}^{\pm 1}]$ is the ring of  the Laurent polynomials in  variables $\mathbf{x}=(x_1,\dots,x_n)$
depending on an initial seed $(\mathbf{x}, B)$.
 The  $\mathscr{A}(\mathbf{x}, B)$  is a commutative algebra with an additive abelian
semigroup consisting of the Laurent polynomials with positive coefficients. 
In particular, it has  an order  satisfying the Riesz interpolation property,  so that  $\mathscr{A}(\mathbf{x}, B)$ becomes  a 
 dimension group  [Effros 1981]  \cite[Theorem 3.1]{E}. 
 Such groups are known to classify via $K$-theory the approximately finite-dimensional (AF) $C^*$-algebras,
 i.e. the direct limits of the matrix  $C^*$-algebras $M_{n_1}(\mathbf{C})\oplus\dots\oplus M_{n_k}(\mathbf{C})$
  [Effros 1981]  \cite{E}. 
 A  cluster $C^*$-algebra  $\mathbb{A}(\mathbf{x}, B)$  is   an  AF-algebra,  
such that $K_0(\mathbb{A}(\mathbf{x}, B))\cong  \mathscr{A}(\mathbf{x}, B)$,
where $\cong$ is an isomorphism of the dimension groups \cite[Section 4.4]{Nik1}.

\bigskip
An annulus in the complex plane will be denoted by 
%**********************************************************************************
\begin{equation}\label{eq1.3}
\mathscr{D}=\{z=x+iy\in\mathbf{C} ~|~ r\le |z|\le R\}. 
\end{equation}
%***********************************************************************************
Recall that the Riemann surfaces  $\mathscr{D}$ and $\mathscr{D}'$  are
conformally equivalent   if and only if     $R/r=R'/r':=t$.  By 
%***********************************************************************************
%\begin{equation}\label{t}
$T_\mathscr{D}=\{t\in \mathbf{R} ~|~ t>1\}$
%\end{equation}
%***********************************************************************************
we understand the Teichm\"uller space of the annulus $\mathscr{D}$. 
The  Penner coordinates  on $T_\mathscr{D}$ are encoded by 
the cluster algebra $\mathscr{A}(\mathbf{x}, B)$, where
%*************************************************************************
\begin{equation}\label{eq1.4}
B=\left(
\begin{matrix}
0 & 2\cr -2 & 0
\end{matrix}
\right),   
\end{equation}
%*************************************************************************
see [Fomin,  Shapiro  \& Thurston  2008]  \cite[Example 4.4]{FoShaThu1}
and [Williams 2014] \cite[Section 3]{Wil1}. 
The corresponding  cluster $C^*$-algebra  $\mathbb{A}(\mathscr{D})$ is given by 
the Bratteli diagram in Figure 1,  which shows the inclusions of the matrix algebras 
 $M_{n_1}(\mathbf{C})\oplus\dots\oplus M_{n_k}(\mathbf{C})$ in the AF-algebra
 $\mathbb{A}(\mathscr{D})$. 
The latter is known as a GICAR (Gauge Invariant Canonical Anticommutation Relations) 
algebra
 [Davidson 1996]  \cite[Example III.5.5]{D} and [Effros 1980]  \cite[p.13(e)]{E}. 
 Moreover, let $M_{2^{\infty}}:=\bigotimes_{i=1}^{\infty} M_2(\mathbf{C})$ be the uniformly hyperfinite (UHF)
 algebra, where $M_2(\mathbf{C})$ is a matrix algebra. Then 
 there exists an embedding of the AF-algebras: 
   %*************************************************************************
\begin{equation}\label{eq1.6}
\mathbb{A}(\mathscr{D})\hookrightarrow  M_{2^{\infty}}. 
\end{equation}
%*************************************************************************

 \bigskip
 On the other hand, the UHF-algebra $M_{2^{\infty}}$ is known in quantum statistical mechanics 
 as a CAR (Canonical Anticommutation Relations) algebra,  which  plays an
  outstanding r\^ole in the theory of subfactors [Jones 1991] \cite[Section 5.6]{J1}.  
  In this note we use (\ref{eq1.6}) and geometry  of
 $\mathbb{A}(\mathscr{D})$   to give a new shorter  proof of  the Jones Index Theorem: 
%****************************************************************************
\begin{theorem}\label{thm1.1}
 There is a subfactor $\mathscr{N}$ of the 
hyperfinite $II_1$ factor $\mathscr{M}$ only if 
$[\mathscr{M}:\mathscr{N}]  ~\in  ~[4, \infty) ~\bigcup  ~\{ 4\cos^2\left({\pi\over n}\right) ~|~n\ge 3\}$. 
\end{theorem}
%*****************************************************************************
The article is organized as follows. Section 2 contains a brief review
of preliminary results.  Theorem \ref{thm1.1} is proved  in Section 3.

%*******************************************************************
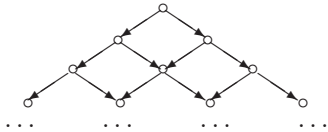
\begin{figure}
%[here]
\begin{picture}(100,100)(0,150)

\put(50,200){\circle{3}}

\put(33,188){\circle{3}}
\put(67,188){\circle{3}}

\put(50,177){\circle{3}}
\put(16,177){\circle{3}}
\put(84,177){\circle{3}}

\put(-1,164){\circle{3}}
\put(34,164){\circle{3}}
\put(68,164){\circle{3}}
\put(103,164){\circle{3}}

%***************************************************

\put(49,199){\vector(-3,-2){15}}
\put(51,199){\vector(3,-2){15}}

%%%%%%%%%%%%%%%%%%%%%%
\put(32,187){\vector(-3,-2){15}}
\put(34,187){\vector(3,-2){15}}

\put(66,187){\vector(-3,-2){15}}
\put(68,187){\vector(3,-2){15}}

%%%%%%%%%%%%%%%%%%%%%%%

\put(14,175){\vector(-3,-2){15}}
\put(17,175){\vector(3,-2){15}}

\put(49,175){\vector(-3,-2){15}}
\put(51,175){\vector(3,-2){15}}

\put(83,175){\vector(-3,-2){15}}
\put(86,175){\vector(3,-2){15}}

\put(-10,155){$\dots$}
\put(27,155){$\dots$}
\put(64,155){$\dots$}
\put(101,155){$\dots$}

\end{picture}
\caption{Bratteli diagram of the cluster $C^*$-algebra $\mathbb{A}(\mathscr{D})$.} 
\end{figure}
%*******************************************************************

%**************************************************************************
\section{Preliminaries}
%***************************************************************************
%**************************************************************************
\subsection{Cluster algebras of rank 2}
%***************************************************************************
Let $x_1$ and $x_2$ be independent variables of a cluster algebra.
For a pair of positive integers $b$ and $c$,  we define elements
$x_i$   by the exchange relations
%*********************************************************************************************
\begin{equation}\label{eq:clusterrelations}
x_{i-1} x_{i+1} =
\left\{
\begin{array}[h]{ll}
1+x_i^b &  \mbox{if} \quad \mbox{$i$ odd,} \\
1+x_i^c  & \mbox{if}  \quad \mbox{$i$ even.}
\end{array}
\right.
\end{equation}
%*********************************************************************************************
By a cluster algebra  rank 2 we denote  the algebra 
$\mathscr{A}(b,c)$ generated by the cluster variables $x_i$
[Sherman \& Zelevinsky 2004]   \cite[Section 2]{SheZe1}. 
Let  $\mathcal{B}$ be a basis of the algebra $\mathscr{A}(b,c)$.
%***************************************************************************
\begin{theorem}\label{thm2.1}
{\bf  (\cite[Theorem 2.8]{SheZe1})}
Suppose that  $b=c=2$ or $b=1$ and  $c=4$. Then  
%*************************************************************************
%\begin{equation}\label{eq3.1}
$\mathcal{B}=\{x_i^px_{i+1}^q ~|~ p,q\ge 0\} ~\bigcup ~\{T_n(x_1x_4-x_2x_3) ~|~ n\ge 1\}$,
%\end{equation}
%************************************************************************* 
where $T_n(x)$ are the Chebyshev polynomials of the first kind.  
\end{theorem}
%****************************************************************************

\bigskip
Let $r<R$ and consider an annulus $\mathscr{D}$ of the form (\ref{eq1.3}) having
one marked point on each boundary component. 
The cluster algebra $\mathscr{A}(b,c)$  associated to an ideal  triangulation 
 of $\mathscr{D}$ is given by the matrix (\ref{eq1.4})  [Fomin,  Shapiro  \& Thurston  2008]  \cite[Example 4.4]{FoShaThu1}.  
The exchange relations in this case can be written as $x_{i-1}x_{i+1}=1+x_i^2$
and  $B'=-B$.  Comparing with  the relations (\ref{eq:clusterrelations}),  we conclude that the 
$\mathscr{A}(b,c)$ is a cluster algebra of rank 2 with $b=c=2$. 
Therefore the basis $\mathcal{B}$ of the cluster algebra $\mathscr{A}(b,c)$ is 
described by Theorem \ref{thm2.1}. 
On the other hand, the cluster algebra  $\mathscr{A}(2,2)$ is known to 
encode  the Penner coordinates on the Teichm\"uller  space 
$T_\mathscr{D}=\{t\in \mathbf{R} ~|~ t>1\}$ of the annulus $\mathscr{D}$
 [Williams 2014] \cite[Section 3]{Wil1}.

Let  $\mathbb{A}(2, 2)$ be an AF-algebra, such that 
$K_0(\mathbb{A}(2, 2))\cong \mathscr{A}(2, 2)$.
The Bratteli diagram of the  cluster $C^*$-algebra $\mathbb{A}(2, 2)$ 
has the form of a  Pascal triangle  shown in 
Figure 1 \cite[Section 4.4]{Nik1}.   Thus  $\mathbb{A}(2, 2)$ is 
a   GICAR  algebra  [Effros 1980]  \cite[p. ~13(e)]{E}.  
Consider a group of the modular automorphisms 
%********************************************************************
\begin{equation}
\sigma_t: \mathbb{A}(2, 2)\to \mathbb{A}(2, 2)
\end{equation}
%**************************************************************
constructed in  \cite[Section 4]{Nik1}. Such a group  is  generated by the  geodesic flow on 
 the Teichm\"uller space  $T_{\mathscr{D}}$, {\it ibid.}

%**************************************************************************
\subsection{Powers state}
%***************************************************************************
Let $M_{2^{\infty}}=\bigotimes_{i=1}^{\infty} M_2(\mathbf{C})$  be the GICAR algebra
[Davidson 1996]\cite[Example III.5.5]{D} and  [Effros 1980]  \cite[p.~13(c1)]{E}.  
For $0<\lambda <1$  and $x_i \in M_2(\mathbf{C}) $  consider the Powers state $\varphi_{\lambda}$
on  the tensor product $M_{2^{\infty}}$ given by the formula:
%************************************************************************************************
\begin{equation}
\varphi_{\lambda}(x_1\otimes\dots\otimes x_n\otimes 1\otimes\dots)
=\prod_{i=1}^n Tr \left({1\over 1+\lambda}
\left(
\begin{matrix} 1 & 0\cr 0 & \lambda \end{matrix} 
\right) x_i
\right).
 \end{equation}
%*********************************************************************************************  
Applying the GNS construction  to the pair $(M_{2^{\infty}},\varphi_{\lambda})$
one gets a factor $R_{\lambda}$.  The product 
%************************************************************************************************
%\begin{equation}
$\left\{\bigotimes_{i=1}^{\infty}\exp \left(\sqrt{-1}\left(\begin{matrix}1 & 0\cr 0 & \lambda\end{matrix}\right)\right)
~|~ 0<\lambda < 1\right\}$
%\end{equation}
%*********************************************************************************************  
gives rise to a group  of  the  modular automorphisms of $R_{\lambda}$,  
see e.g. [Jones 1991] \cite[Section 1.10]{J1}.

\medskip
The GICAR algebra $\mathbb{A}(2, 2)$  embeds  into the factor $R_{\lambda}$
[Davidson 1996]  \cite[Example III.5.5]{D}.   Moreover,  a restriction of the modular 
 automorphisms of $R_{\lambda}$ coincides with the $\sigma_t: \mathbb{A}(2, 2)\to \mathbb{A}(2, 2)$
 constructed in   \cite[Section 4]{Nik1}.

%**************************************************************************
\subsection{Basic construction}
%***************************************************************************
 Denote by $e_{ij}$ the matrix units of the algebra $M_2(\mathbf{C})$. 
Then 
%*************************************************************************
%\begin{equation}
$e_t={1\over 1+t} (e_{11}\otimes e_{11}+te_{22}\otimes e_{22}+\sqrt{t}
(e_{12}\otimes e_{21}+e_{21}\otimes e_{12}))$
%\end{equation}
%*************************************************************************
is a projection  of the algebra $M_2(\mathbf{C})\otimes M_2(\mathbf{C})$
for each  $t\in\mathbf{R}$.  Proceeding by induction, one can define projections
$e_i(t)=\theta^i(e_t)\in M_{2^i}$, where $\theta$ is the shift automorphism of the 
$UHF$-algebra $M_{2^{\infty}}$ and $M_{2^i}$ is the $i$-th element of $\bigotimes_{i=1}^{\infty} M_2(\mathbf{C})$.  The  $e_i:=e_i(t)$  satisfy the following relations
%*******************************************************************
\begin{equation}\label{eq1.8}
\left\{
\begin{array}{ccc}
e_i e_j &=& e_j e_i, \quad\hbox{if} \quad |i-j|\ge 2\\
e_i e_{i\pm 1} e_i &=&{t\over (1+t)^2} e_i,
\end{array}
\right.
\end{equation}
%*****************************************************************
so that $Tr~(xe_{n+1})= [\mathscr{M}:\mathscr{N}]^{-1} ~Tr ~(x)$
[Jones 1991]  \cite[Section 5.6]{J1}. 
The $e_i(t)$ generate 
 a subfactor $\mathscr{N}$
 of the type II von Neumann algebra $\mathscr{M}$,  such that  
 %****************************************************************************
 \begin{equation}\label{eq1.9}
 [\mathscr{M}:\mathscr{N}]^{-1}={t\over (1+t)^2}.
\end{equation}
%***********************************************************************

%**************************************************************************
\section{Proof of Theorem \ref{thm1.1}}
%***************************************************************************
We shall use  a simple  analysis  of the cluster algebra  $\mathscr{A}(\mathscr{D})\cong K_0(\mathbb{A}(\mathscr{D}))$
using the Sherman-Zelevinsky Theorem.  
Namely,  such an algebra has a canonical basis
of the form
%*************************************************************************
\begin{equation}\label{eq3.1}
\mathcal{B}=\{x_i^px_{i+1}^q ~|~ p,q\ge 0\} ~\bigcup ~\{T_n(x_1x_4-x_2x_3) ~|~ n\ge 1\},
\end{equation}
%************************************************************************* 
where $T_n(x)$ are the Chebyshev polynomials of the first kind,
see Theorem \ref{thm2.1}. 
We  split the proof in two lemmas corresponding (roughly) to the cases
 $|\mathcal{B}|=\infty$ and  $|\mathcal{B}|<\infty$, respectively.

\bigskip
%********************************************************************************
\begin{lemma}\label{lem3.1}
There exists a subfactor $\mathscr{N}$ of the hyperfinite type $II_1$ factor $\mathscr{M}$
whenever
$[\mathscr{M}:\mathscr{N}]\in (4, \infty)$.  
\end{lemma}
%**********************************************************************************
\begin{proof}
(i)  Let us return to the inclusion (\ref{eq1.6}) and consider the Powers state $\varphi_{\lambda}$
on $M_{2^{\infty}}$.  The Powers modular automorphism of the factor $R_{\lambda}$ 
induces a modular automorphism $\sigma_t: ~\mathbb{A}(\mathscr{D})\to \mathbb{A}(\mathscr{D})$. 
The Penner coordinate $t=R/r>1$ on $T_{\mathscr{D}}$ and 
the Powers  parameter $0<\lambda<1$ are related by the formula:
%*************************************************************************
\begin{equation}\label{scale}
t={1\over 2}\left(\lambda+{1\over \lambda}\right). 
\end{equation}
%************************************************************************* 

In other words, the Penner coordinates give the Powers states, i.e.
for each $t>1$ the evaluation map produces a positive homomorphism
of $K_0(\mathbb{A}(\mathscr{D}))$ to $\mathbf{R}$, which correlates 
with a trace on the GICAR algebra $\sigma_t(\mathbb{A}(\mathscr{D}))$.

\bigskip
(ii) If $|\mathcal{B}|=\infty$, then the Bratteli diagram of  $\mathbb{A}(\mathscr{D})$
(Figure 1) is an infinite tower.  The hyperfinite type $II_1$ factor $\mathscr{M}$ 
is obtained from a factor $\mathscr{N}$ by adjoining the Jones projections $e_i(t)$
using the basic construction (Section 2.3).  The Penner coordinate $t>1$ on
$T_{\mathscr{D}}$ corresponds to the values of  index $[\mathscr{M}:\mathscr{N}]={(1+t)^2\over t}>4$
in view of formula  (\ref{eq1.9}).   In other words, $[\mathscr{M}:\mathscr{N}]\in (4, \infty)$. 
Lemma \ref{lem3.1} is proved. 
\end{proof}

\bigskip
%*************************************************************************
\begin{lemma}\label{lem3.2}
There exists a subfactor $\mathscr{N}$ of the hyperfinite type $II_1$ factor $\mathscr{M}$
whenever $[\mathscr{M}:\mathscr{N}]\in \{ 4\cos^2\left({\pi\over n}\right) ~|~n\ge 3\}\cup \{4\}$. 
\end{lemma}
%****************************************************************************
\begin{proof}
(i) Recall that the Chebyshev polynomials satisfy the following relations:
%*************************************************************************
\begin{equation}\label{eq21}
T_0 =1 \quad\hbox{and}  \quad  T_n \left[{1\over 2}(t+t^{-1})\right]={1\over 2}(t^n+t^{-n}). 
\end{equation}
%************************************************************************* 

In view of \ref{thm2.1}, we choose
${1\over 2}(t+t^{-1})=x_1x_4-x_2x_3$.
(Such a parametrization  is always possible since the Penner coordinates 
[Williams 2014]  \cite[Section 3.2]{Wil1}
on $T_\mathscr{D}$
are given by the cluster $(x_1,x_2)$,  where each $x_i$   is a function of   $t$.)

The  exchange relations (\ref{eq1.2}) for   $\mathscr{A}(\mathscr{D})$  
can be written as $x_{i-1}x_{i+1}=x_i^2+1$. 
It is easy  to calculate that
$x_1x_4-x_2x_3={x_1^2+1+x_2^2\over x_1x_2}$.  
An explicit  resolution of cluster variables $x_1$ and $x_2$ is given 
by the formulas:
%*******************************************************************
\begin{equation}\label{eq28}
\left\{
\begin{array}{lll}
x_1&=& {\sqrt{2}\over 2}\sqrt{t^2+t\sqrt{t^2-16}}   \\
x_2 &=& {\sqrt{2}\over 2}\sqrt{t^2-t\sqrt{t^2-16}}    
\end{array}
\right.
\end{equation}
%*****************************************************************
The reader can verify,  that equations (\ref{eq28})  imply $x_1x_4-x_2x_3= {1\over 2}(t+t^{-1})$.
The parametrization of the ordered $K_0$-group of the GICAR algebra  $\mathbb{A}(\mathscr{D})$ 
in this case differs from (\ref{scale})  in the sense that $t$ is allowed to be a complex number. 
As we shall see, such an extension does not affect  the property of the index to be a real number. 
The compatibility of traces under the embedding (\ref{eq1.6}) is preserved.

\bigskip
(ii)   If $|\mathcal{B}|<\infty$, then the Bratteli diagram of  $\mathbb{A}(\mathscr{D})$
(Figure 1) is a finite tower.  
In particular,    the formulas (\ref{eq3.1}) and  (\ref{eq21}) imply
%*************************************************************************
\begin{equation}\label{eq29}
T_n(x_1x_4-x_2x_3)=T_0=1
\end{equation}
%************************************************************************* 
for some integer $n\ge 1$.   But  $x_1x_4-x_2x_3= {1\over 2}(t+t^{-1})$
and using formula (\ref{eq21})  for the Chebyshev polynomials, one gets
an equation
%*************************************************************************
\begin{equation}\label{eq30}
t^n+t^{-n}=2
\end{equation}
%************************************************************************* 
for (possibly complex)  values of $t$.  Since  (\ref{eq30}) 
is equivalent to the equation $t^{2n}-2t^n+1=(t^n-1)^2=0$,  one gets the 
$n$-th root of unity
%*************************************************************************
\begin{equation}\label{eq31}
t\in \{e^{2\pi i\over n} ~|~ n\ge 1\}. 
\end{equation}
%************************************************************************* 
The value
%*************************************************************************
\begin{equation}\label{eq32}
 [\mathscr{M}:\mathscr{N}]={(1+t)^2\over t}={1\over t}+2+t=2\left[\cos~\left({2\pi\over n}\right)+1\right]=
4\cos^2\left({\pi\over n}\right)
\end{equation}
%************************************************************************* 
is a real number.  We must exclude the case $n=2$ corresponding to the value $t=-1$,  because 
otherwise one gets a division by zero in   (\ref{eq1.8}).
  Lemma \ref{lem3.2} is proved. 
\end{proof}

\bigskip
Theorem \ref{thm1.1} follows from  lemmas \ref{lem3.1}   and \ref{lem3.2}.

\bibliographystyle{amsplain}

%**********************************************************

\end{document}